\documentclass[11pt]{article}

\usepackage{amssymb,amsmath,amsfonts,amsthm}
\usepackage{latexsym}
\usepackage{graphics}
\usepackage{indentfirst}
\usepackage{hyperref}
\usepackage[capitalise, noabbrev, nameinlink]{cleveref}
\usepackage{comment}
\usepackage[shortlabels]{enumitem}
\usepackage[english]{babel}
\usepackage{tikz}
\usepackage{esdiff}

\definecolor{azure(colorwheel)}{rgb}{0.0, 0.5, 1.0}

\usetikzlibrary{arrows.meta}

\newtheorem{innercustomthm}{Theorem}
\newenvironment{customthm}[1]
  {\renewcommand\theinnercustomthm{#1}\innercustomthm}
  {\endinnercustomthm}

\newtheorem{innercustompro}{Proposition}
\newenvironment{custompro}[1]
  {\renewcommand\theinnercustompro{#1}\innercustompro}
  {\endinnercustompro}
\newtheorem*{thm*}{Theorem}
\newtheorem{thm}{Theorem}
\newtheorem{lem}[thm]{Lemma}
\newtheorem{pro}[thm]{Proposition}
\newtheorem{obs}[thm]{Observation}

\newtheorem{ques}[thm]{Question}

\newcommand{\N}{\mathbb{N}}

\newcommand{\col}{\mathrm{col}}

\allowdisplaybreaks

\begin{document}

\title{Maximizing Satisfied Vertex Requests in List Coloring}

\author{ 
Timothy Bennett \thanks{Department of Mathematics and Statistics, University of South Alabama, Mobile, AL, USA (tdb1923@jagmail.southalabama.edu)}
\and
Michael C. Bowdoin \thanks{Mitchell College of Business, University of South Alabama, Mobile, AL, USA (mb2139@jagmail.southalabama.edu)}
\and
Haley Broadus \thanks{Department of Mathematics and Statistics, University of South Alabama, Mobile, AL, USA (hlb2121@jagmail.southalabama.edu)}
\and
Daniel Hodgins \thanks{Department of Mathematics \& Statistics, Auburn University, Auburn, AL, USA (djh0081@auburn.edu)}
\and
Jeffrey A. Mudrock \thanks{Department of Mathematics and Statistics, University of South Alabama, Mobile, AL, USA (mudrock@southalabama.edu)}
\and
Adam K. Nusair \thanks {College of Engineering, University of South Alabama, Mobile, AL, USA, (Adamnusair913@gmail.com)}
\and
Gabriel Sharbel \thanks{School of Computing, University of South Alabama, Mobile, AL, USA (gss2121@jagmail.southalabama.edu)}
\and 
Joshua Silverman \thanks{Department of Mathematics and Statistics, University of South Alabama, Mobile, AL, USA (jrs2123@jagmail.southalabama.edu)}
}
\maketitle

\begin{abstract}

Suppose $G$ is a graph and $L$ is a list assignment for $G$.  A \emph{request} of $L$ is a function $r$ with nonempty domain $D\subseteq V(G)$ such that $r(v) \in L(v)$ for each $v \in D$.  The triple $(G,L,r)$ is $\epsilon$-satisfiable if there exists a proper $L$-coloring $f$ of $G$ such that $f(v) = r(v)$ for at least $\epsilon|D|$ vertices in $D$.  We say $G$ is $(k, \epsilon)$-flexible if $(G,L',r')$ is $\epsilon$-satisfiable whenever $L'$ is a $k$-assignment for $G$ and $r'$ is a request of $L'$.  It is known that a graph $G$ is not $(k, \epsilon)$-flexible for any $k$ if and only if $\epsilon > 1/ \rho(G)$ where $\rho(G)$ is the Hall ratio of $G$. The \emph{list flexibility number} of a graph $G$, denoted $\chi_{\ell flex}(G)$, is the smallest $k$ such that $G$ is $(k,1/ \rho(G))$-flexible.  A fundamental open question on list flexibility numbers asks: Is there a graph with list flexibility number greater than its coloring number?

In this paper, we show that the list flexibility number of any complete multipartite graph $G$ is at most the coloring number of $G$.  We also initiate the study of \emph{list epsilon flexibility functions} of complete bipartite graphs which was first suggested by Kaul, Mathew, Mudrock, and Pelsmajer in 2024.  Specifically, we completely determine the list epsilon flexibility function of $K_{m,n}$ when $m \in \{1,2\}$ and establish some additional bounds for small $m$.  Our proofs reveal a connection to list coloring complete bipartite graphs with asymmetric list sizes which is a topic that was explored by Alon, Cambie, and Kang in 2021.

\medskip



\end{abstract}

\section{Introduction}\label{intro}
In this paper all graphs are nonempty, finite, simple graphs.  Generally speaking, we follow West~\cite{W01} for terminology and notation.  The set of natural numbers is $\N = \{1,2,3, \ldots \}$.  For $m \in \N$, we write $[m]$ for the set $\{1, \ldots, m \}$.  We write $K_{l,n}$ for complete bipartite graphs with partite sets of size $l$ and $n$.  More generally, when $k\geq2$ and $n_1, n_2, \ldots, n_k \in \mathbb{N}$, we use $K_{n_1, n_2, \ldots, n_k}$ for complete $k$-partite graphs with partite sets of size $n_1, n_2, \ldots,$ and $n_k$.  For a graph $G$, we use $d_G(v)$ (or $d(v)$ when $G$ is clear from context) for the degree of a vertex $v$ in $G$, and we use $\Delta(G)$ for the maximum degree of the vertices in $G$.  If $S \subseteq V(G)$, we use $G[S]$ for the subgraph of $G$ induced by $S$.  We write $H \subseteq G$ when $H$ is a subgraph of $G$.   

In this paper we study flexible list coloring of complete multipartite graphs with a special emphasis on complete bipartite graphs.  We begin by reviewing list coloring and introducing flexible list coloring.

\subsection{List Coloring} \label{basic}

For classical vertex coloring of graphs, we wish to color the vertices of a graph $G$ with up to $m$ colors from $[m]$ so that adjacent vertices receive different colors, a so-called \emph{proper $m$-coloring}.  The \emph{chromatic number} of a graph is a well studied graph invariant. Denoted $\chi(G)$, the chromatic number of graph $G$ is the smallest $m$ such that $G$ has a proper $m$-coloring.  

List coloring is a well-known variation on classical vertex coloring that was introduced independently by Vizing~\cite{V76} and Erd\H{o}s, Rubin, and Taylor~\cite{ET79} in the 1970s.  For list coloring, we associate a \emph{list assignment} $L$ with a graph $G$ such that each vertex $v \in V(G)$ is assigned a list of available colors $L(v)$ (we say $L$ is a list assignment for $G$).  We say $G$ is \emph{$L$-colorable} if there is a proper coloring $f$ of $G$ such that $f(v) \in L(v)$ for each $v \in V(G)$ (we refer to $f$ as a \emph{proper $L$-coloring} of $G$).  A list assignment $L$ for $G$ is called a \emph{$k$-assignment} if $|L(v)|=k$ for each $v \in V(G)$.  The \emph{list chromatic number} of a graph $G$, denoted $\chi_\ell(G)$, is the smallest $k$ such that $G$ is $L$-colorable whenever $L$ is a $k$-assignment for $G$.  We also say that $G$ is \emph{$t$-choosable} if $\chi_{\ell}(G) \leq t$.  It is easy to show that for any graph $G$, $\chi(G) \leq \chi_\ell(G)$.  Moreover, it is well-known that the gap between the chromatic number and list chromatic number of a graph can be arbitrarily large since $\chi_{\ell}(K_{n,t}) = n+1$ whenever $t \geq n^n$ (see~\cite{ET79} for additional discussion).  When we are considering list assignments for $G$ that assign lists of various sizes to the vertices of $G$, we need another notion.  If $f: V(G) \rightarrow \N$, we say that $G$ is \emph{$f$-choosable} if $G$ is $L$-colorable whenever $L$ is a list assignment for $G$ satisfying $|L(v)|=f(v)$ for each $v \in V(G)$.  If $f: V(G) \rightarrow \N$ and $g: V(G) \rightarrow \N$ satisfy $g(v) \geq f(v)$ for each $v \in V(G)$, it is immediately clear that $G$ is $f$-choosable only if $G$ is $g$-choosable.    

The \emph{coloring number} of a graph $G$, $\col (G)$, is the smallest integer $d$ such that there exists an ordering, $v_1,\ldots, v_n$, of the elements of $V(G)$ such that $v_i$ has at most $d-1$ neighbors preceding it in the ordering.  For example, when $k \geq 2$ and  $n_1, n_2, \ldots, n_k \in \mathbb{N}$ satisfy $n_1 \leq n_2\leq \cdots \leq n_k$, $\col(K_{n_1,n_2,\ldots, n_k}) = 1 + \sum_{i=1}^{k-1} n_i$. It is also easy to see that $\chi_{\ell} (G) \leq \col(G)$.

Since complete bipartite graphs play an important role in this paper, it is worth noting that finding the exact list chromatic number of complete bipartite graphs is notoriously difficult.  While it has been known since the 1970s that for $1 \leq m \leq n$, $K_{m,n}$ is 2-choosable if and only if $m=1$ and $n \in \N$ or $m=2$ and $n \leq 3$, the 3-choosable complete bipartite graphs were not fully characterized until 1995~\cite{O95} about 20 years after the introduction of list coloring (see section 8.2 in~\cite{W20} for some discussion).  Specifically, $K_{2,n}$ is 3-choosable for each $n \in \N$, and for $3 \leq m \leq n$, $K_{m,n}$ is 3-choosable if and only if $m=3$ and $n \leq 26$, $m=4$ and $n \leq 20$, $m=5$ and $n \leq 12$, or $m=6$ and $n \leq 10$.

Some notation from~\cite{AC21} will be used throughout the paper.  Suppose $G$ is a bipartite graph with bipartition $A$, $B$.  A list assignment $L$ for $G$ is called a \emph{$(k_A,k_B)$-assignment} if $|L(v)|=k_A$ for each $v \in A$ and $|L(v)|=k_B$ for each $v \in B$.  We say that $G$ is $(k_A,k_B)$-choosable if $G$ is $L$-colorable whenever $L$ is a $(k_A,k_B)$-assignment for $G$.  In the case that $G = K_{m,n}$, when we say $G$ is \emph{$(a,b)$-choosable} we mean that $G$ is $(k_A,k_B)$-choosable where $A$ is the partite set of $G$ of size $m$, $B$ is the partite set of $G$ of size $n$, $k_A = a$, and $k_B=b$.  We adopt an analogous convention for \emph{$(a,b)$-assignment of $G$}.  

\subsection{Flexible List Coloring}

Flexible list coloring was introduced by Dvo\v{r}\'{a}k, Norin, and Postle in~\cite{DN19} in order to address a situation in list coloring where we still seek a proper list coloring, but each vertex may have a preferred color assigned to it, and for those vertices we wish to color as many of them with their preferred colors as possible.  Specifically, suppose $G$ is a graph and $L$ is a list assignment for $G$.  A \emph{request of $L$} is a function $r$ with nonempty domain $D\subseteq V(G)$ such that $r(v) \in L(v)$ for each $v \in D$.  For each $v \in D$ we say that the pair $(v,r(v))$ is the \emph{vertex request for $v$} with respect to $r$ (we omit the phrase with respect to $r$ when $r$ is clear from context).  For any $\epsilon \in [0,1]$, the triple $(G,L,r)$ is \emph{$\epsilon$-satisfiable} if there exists a proper $L$-coloring $f$ of $G$ such that $f(v) = r(v)$ for at least $\epsilon|D|$ vertices in $D$.  When $f(v) = r(v)$ for some $v \in D$ we say $f$ \emph{satisfies} the vertex request of $r$ for $v$.  

We say that the pair $(G,L)$ is \emph{$\epsilon$-flexible} if $(G,L,r)$ is $\epsilon$-satisfiable whenever $r$ is a request of $L$.  Finally, we say that $G$ is \emph{$(k, \epsilon)$-flexible} if $(G,L)$ is $\epsilon$-flexible whenever $L$ is a $k$-assignment for $G$.  Note that if $G$ is $k$-choosable, then $G$ is $(k,0)$-flexible.  The following observation is also immediate.    

\begin{obs} \label{obs: basic}
Suppose $G$ is $(k,\epsilon)$-flexible for some $\epsilon \in [0,1]$. Then, the following statements hold. 
\\
(i) $G$ is $(k',\epsilon')$-flexible for any $k'\ge k$ and $0 \leq \epsilon'\le \epsilon$. 
\\
(ii) Any subgraph $H$ of $G$ is $(k,\epsilon)$-flexible.
\\
(iii) $G$ is $k$-choosable.
\end{obs}

\subsection{The Hall Ratio, List Flexibility Number, and List Epsilon Flexibility Function}

For a graph $G$, what is the largest $\epsilon$ so that $G$ is $(k,\epsilon)$-flexible for some $k$?  Suppose $r$ is a request with domain $D$ of some list assignment.  It is possible that $r(v)$ is the same color for all $v\in D$; for example, let $L$ be the $k$-assignment such that $L(v)=[k]$ for all $v\in V(G)$, and let $r(v)=1$ for all $v\in D$.  Then at most $\alpha(G[D])$ vertices in $D$ will have their request satisfied.  So, when $G$ is $(k,\epsilon)$-flexible, $\epsilon$ must satisfy $\epsilon |D|\le \alpha(G[D])$ for every nonempty $D\subseteq V(G)$ regardless of $k$.  In particular, $\epsilon\le \min_{\emptyset\not=D\subseteq V(G)} \alpha(G[D])/|D|$ for any $(k,\epsilon)$-flexible graph $G$.

The \emph{Hall ratio} of a graph $G$, denoted $\rho(G)$, is 
$$\rho(G) = \max_{\emptyset\not= H \subseteq G} \frac{|V(H)|}{\alpha(H)}.$$  
The Hall ratio was first studied in 1997 by Hilton, Johnson Jr., and Leonard~\cite{HJ97} under the name fractional Hall-condition number. Since then the Hall ratio has received much attention due to its connection with both list and fractional coloring (e.g., see~\cite{BL22, CC21, CJ00, DM20}).

Note that among subgraphs $H$ of $G$ with fixed vertex set $D$, $\alpha(H)$ is minimized by the induced subgraph $H=G[D]$.  Therefore, $\min_{\emptyset\not=D\subseteq V(G)} \alpha(G[D])/|D| = 1/\rho(G)$.  Moreover, this bound on feasible $\epsilon$ for $(k,\epsilon)$-flexibility is attainable.

\begin{pro} [\cite{KM23}] \label{pro: fundamental}
Graph $G$ is $(\Delta(G)+1,\epsilon)$-flexible if and only if $\epsilon\le 1/\rho(G)$.
\end{pro}

\cref{pro: fundamental} now allows us to present an important definition.  The \emph{list flexibility number} of $G$, denoted $\chi_{\ell flex}(G)$, is the smallest $k$ such that $G$ is $(k,1/\rho(G))$-flexible.  The following result is now immediate.

\begin{pro} [\cite{KM23}] \label{pro: inequalities}
For any graph $G$,
$$\chi(G) \leq \chi_{\ell}(G) \leq \chi_{\ell flex}(G) \leq \Delta(G) + 1.$$
\end{pro}

Recall that the coloring number of a graph is a natural upper bound for its list chromatic number.  So, it is natural to wonder if this upper bound extends to the list flexibility number. Interestingly, there are no known graphs $G$ for which the list flexibility number of $G$ exceeds the coloring number of $G$. This leads to a fundamental open question on the list flexibility number which we will keep in mind for the remainder of the paper (Question~\ref{ques: coloringnumber} below is Question~10 in~\cite{KM23}).

\begin{ques} [\cite{KM23}] \label{ques: coloringnumber}
    Does there exist a graph $G$ satisfying $\chi_{\ell flex}(G) > \col (G)$?
\end{ques}

For each graph $G$ and fixed $t \in \N$, suppose we are interested in the maximum possible $\epsilon$ such that $G$ is $(t,\epsilon)$-flexible.  The \emph{list epsilon flexibility function} for $G$, denoted $\epsilon_\ell(G,t)$, is the function that maps each $t \geq \chi_{\ell}(G)$ to the largest $\epsilon \in [0,1]$ such that $G$ is $(t,\epsilon)$-flexible.  The study of the list epsilon flexibility function was first suggested in~\cite{KM23}, and the following question served as one of the motivations for this paper.   

\begin{ques} \label{ques: completebipartite}
Suppose $1 \leq m \leq n$ and $G = K_{m,n}$.  What is a formula for $\epsilon_\ell(G,t)$?  
\end{ques}

Note that for any graph $G$ and $t \geq \chi_{\ell}(G)$, $\epsilon_{\ell}(G,t)=a/b$ for some integers $0\le a\le b\le |V(G)|$ and $\epsilon_\ell(G,t) \leq 1/\rho(G)$.  Furthermore, the list epsilon flexibility function of $G$ is eventually constant since $\epsilon_\ell(G,t) = 1/ \rho(G)$ whenever $t \geq \chi_{\ell flex}(G)$.  Since $\chi_{\ell}(K_{1,n}) = \chi_{\ell flex}(K_{1,n})=2$ (see~\cite{BM22}), it follows that $\epsilon_\ell(K_{1,n},t) = 1/2$ whenever $t \geq 2$.  In this paper we make some modest progress on Question~\ref{ques: completebipartite}.  Specficially, we completely determine $\epsilon_\ell(G,t)$ when $G=K_{2,n}$, and we come very close to completely determining $\epsilon_\ell(G,t)$ when $G=K_{3,n}$ (see Theorems~\ref{cor: K2n} and~\ref{thm: K3n} below).

\subsection{Outline of Results}

In Section~\ref{Big}, we show the answer to Question~\ref{ques: coloringnumber} is no when we restrict our attention to complete multipartite graphs.

\begin{thm}\label{thm: awesome_sauce}
Suppose $k\geq2$. Suppose $n_1, n_2, \ldots, n_k \in \mathbb{N}$ satisfy  $n_1 \leq n_2\leq \cdots \leq n_k$, and $G=K_{n_1,n_2,\ldots, n_k}$.  Then, $\chi_{\ell flex}(G)\leq n_1+\cdots+ n_{k-1} +1$.
\end{thm}

Suppose $G=K_{n_1,n_2,\ldots, n_k}$. Note that since $\rho(G)=k$, Theorem~\ref{thm: awesome_sauce} tells us that $\epsilon_{\ell}(G,t) = 1/k$ whenever $t \geq n_1+\cdots+ n_{k-1} +1$.  

In Section~\ref{small} we turn our attention to making progress on Question~\ref{ques: completebipartite}.  Suppose $G=K_{m,n}$ with $1\le m\le n$.  We know that $\epsilon_\ell(G,t) \le 1/\rho(G)=1/2$ whenever $t \geq \chi_{\ell}(G)$, and we know that $\epsilon_\ell(G,t) = 1/2$ whenever $t\ge \chi_{\ell flex}(G)$.  Theorem~\ref{thm: awesome_sauce} tells us $\epsilon_\ell(G,t) = 1/2$ whenever $t\ge \col(G)=m+1$. We use this fact along with some other basic ideas to completely determine $\epsilon_{\ell}(G,t)$ when $G = K_{2,n}$.
\begin{thm} \label{cor: K2n}
Suppose $G = K_{2,n}$.  Then,
\[\epsilon_{\ell}(G,t)=
\begin{cases}
1/2 \text{ if } n=1 \text{ and } t\geq2
\\0 \text{ if } n\in \{2,3\} \text{ and } t=2 
\\1/2 \text{ otherwise}
\end{cases}
.\]
\end{thm}
After proving Theorem~\ref{cor: K2n}, we turn our attention to $\epsilon_{\ell}(G,t)$ when $G = K_{3,n}$ and $n \geq 3$.  Specifically, we prove the following result.

\begin{thm} \label{thm: K3n}
Suppose $G = K_{3,n}$ and $n \geq 3$ and $n\in \mathbb{N}$.  Then, the following statements hold. 
\begin{enumerate}[label=(\roman*)]
\item $\epsilon_{\ell}(G,t)= 1/2$ whenever $t\geq 4$.  
\item $\epsilon_{\ell}(G,3)= 1/2$ when $n=3$. 
\item $1/3 \leq \epsilon_{\ell}(G,3)\leq 1/2$ whenever $4 \leq n \leq 6$. 
\item $\epsilon_{\ell}(G,3) = 1/3$ whenever $7 \leq n \leq 8$. 
\item $\epsilon_{\ell}(G,3)= 0$ whenever $9\leq n \leq 26$.
\end{enumerate}
\end{thm}
So, when $G=K_{3,n}$ and $n \geq 3$, we know the value of $\epsilon_{\ell}(G,t)$ unless $n \in \{4,5,6\}$ and $t=3$.  It is also worth mentioning that we give a complete characterization of $(3,2)$-choosable complete bipartite graphs to help with the proof of Theorem~\ref{thm: K3n}.  We state it here since it may be of independent interest.  
\begin{thm} \label{thm: 3_2}
Suppose $G=K_{m,n}$ where $m,n \in \mathbb{N}$ satisfy $m \leq n$.
\begin{enumerate}[label=(\roman*)]
    \item We have $G$ is $(3,2)$-choosable if and only if $m=1$ and $n \in \mathbb{N}$, $m=2$ and $n \leq 8$, $m=3$ and $n \leq 6$, or $m=4$ and $n=4$. 
    \item We have $G$ is $(2,3)$-choosable if and only if $m \in [2]$ and $n \in \mathbb{N}$, $m=3$ and $n \leq 7$, or $m=4$ and $n \leq 5$.
\end{enumerate}
\end{thm}
Theorem~\ref{thm: 3_2} and the characterization of 3-choosable graphs mentioned in Subsection~\ref{basic} also allows us to prove the following result.
\begin{pro} \label{pro: forfree}
    Suppose $G=K_{m,n}$ and $m \leq n$. Then, 
    $\epsilon_{\ell}(G,3) = 0$ when $m=4$ and $7 \leq n \leq 20$, $m=5$ and $n \leq 12$, or $m=6$ and $n \leq 10$.  When $m=4$ and $n \leq 6$, $1/|V(G)| \leq \epsilon_{\ell}(G,3) \leq 1/2$.
\end{pro}

\section{Complete Multipartite Graphs} \label{Big}

Suppose $n_1, n_2, \ldots, n_k \in \mathbb{N}$ satisfy  $n_1 \leq n_2\leq \cdots \leq n_k$, and $G=K_{n_1,n_2,\ldots, n_k}$.  It is easy to prove that $\rho(G) = k$.  We are now ready to present a proof of Theorem~\ref{thm: awesome_sauce}.

\begin{proof}
The proof is by induction on $k$.  Throughout the proof suppose  the partite sets of $G$ are $X_1, X_2, \ldots, X_k$, and suppose $X_1 = \{x_1, \ldots, x_{n_1}\}$.

We begin by proving the result when $k=2$.  Suppose $L$ is an arbitrary $(n_{1}+1)$-assignment of $G$ and $r$ is a request of $L$ with nonempty domain $D$.  Let $ D_i = X_{i} \cap D $ for each $i\in [2]$, and let $P = \bigcup_{j=1}^{n_1} L(x_{j})$.  Let $C : P  \rightarrow (\N \cup \{0\})$ be the function given by $C(z) = |r^{-1}(z) \cap D_{2}|$.  Name the colors of $P$ so that $ P = \{c_1,\ldots,c_p\}$ and $C(c_1) \geq C(c_2) \geq \cdots \geq C(c_p)$. Note that $n_1 +1 \leq p \leq n_1(n_1+1)$, and let $W = D_{2} - r^{-1}(P)$.  

Now, we will construct two proper $L$-colorings of $G$: $f$ and $g$, and we will show that at least one of these proper $L$-colorings satisfies at least $|D|/2$ of the vertex requests.  We construct $f$ as follows. For each $j \in[n_{1}]$, color $x_{j}$ with the element of $L(x_{j}) \cap \{c_{n_1+1},\ldots,c_p\}$ of highest index (this is possible since $|L(x_j)|=n_1+1$).  Finally, complete $f$ by greedily coloring the vertices in $X_{2}$ in any order, satisfying vertex requests when possible. Since no colors in $r(W) \cup \{c_1,\dots,c_{n_{1}}\}$ were used to color the vertices in $X_1$, $f$ satisfies at least $|W|+ \sum_{i=1}^{n_1} C(c_i)$ vertex requests.

Construct $g$ as follows. For each $x\in D_{1}$, let $g(x) = r(x)$. For each $x\in X_1- D_1$, let $g(x)$ be an element of $L(x)$. Complete $g$ by greedily coloring the vertices in $X_2$ in any order, satisfying vertex requests when possible. Clearly, $g$ satisfies $|W|+ |D_{1}| + \sum_{c \in P-g(X_1)} C(c)$ vertex requests which means $g$ satisfies at least $|W|+ |D_{1}| + \sum_{i=n_1+1}^{p} C(c_i)$ vertex requests.  

For the sake of contradiction, assume $|W|+ |D_{1}| + \sum_{i=n_1+1}^{p} C(c_i)<|D|/2$ and  $|W|+ \sum_{i=1}^{n_1} C(c_i) <|D|/2$. Since $|W|+ |D_1|+ \sum_{i=1}^p C(c_i)= |D|$, adding these inequalities yields $|W|< 0$ which is a contradiction. This means $f$ or $g$ is a proper $L$-coloring of $G$ that satisfies at least $|D|/2$ of the vertex requests.  Consequently, $\chi_{\ell flex}(G) \leq n_1+1$.

Next, assume $k \geq 3$ and the desired result holds for all integers greater than one and less than $k$.  Suppose $L$ is an arbitrary $(n_{1}+\cdots+ n_{k-1} +1)$-assignment of $G$ and $r$ is a request of $L$ with nonempty domain $D$.   Let $s=\sum_{j=1}^{k-1} n_{j}$, $D_i = X_{i} \cap D $ for each $i\in [k]$, and $P = \bigcup_{j=1}^{n_1} L(x_{j})$.  Let $D' = \bigcup_{j=2}^k D_j$, and let $C : P  \rightarrow (\N \cup \{0\})$ be the function given by $C(z) = |r^{-1}(z) \cap D'|$.  Name the colors of $P$ so that $P = \{c_1,\ldots,c_p\}$ and $C(c_1) \geq C(c_2) \geq \cdots \geq C(c_p)$. Note that $s +1 \leq p \leq n_1(s+1)$, and let $W = D' - r^{-1}(P)$. 

Now, we will construct two proper $L$-colorings of $G$: $f$ and $g$, and we will show that at least one of these proper $L$-colorings satisfies at least $|D|/k$ of the vertex requests.  We construct $f$ as follows. For each $j \in[n_{1}]$, color $x_{j}$ with the element of $L(x_{j}) \cap \{c_{s+1},\ldots,c_p\}$ of highest index.  Let $z^{(j)} = f(x_{j})$ for each $j \in [n_1]$.  Suppose $Z=\{z^{(j)}: j \in [n_1]\}$, and let $z=|Z|$. Note that $z\in[n_1]$. Since $Z \subseteq \{c_{s+1},\ldots,c_p\}$, $s+z\leq p$.  Next, let $G' = G - X_1$ and $L'(v) = L(v) - Z$ for each $v \in V(G')$.  Also let $r'$ be the function $r$ with its domain restricted to $D' - r^{-1}(Z)$.  Clearly $r'$ is a request of $L'$, $G'= K_{n_2,\dots,n_k}$, and $|L'(v)| \geq s + 1 - n_1 = n_2+\cdots+n_{k-1} + 1$ for each $v \in V(G')$.

So, the induction hypothesis tells us that we can complete $f$ according to a proper $L'$-coloring of $G'$ that satisfies at least  $|D' - r^{-1}(Z)|/(k-1)$ of the vertex requests of $r'$.  Note $\left|D' - r^{-1}(Z)\right| = |W|+\sum_{i=1}^p C(c_i) - \sum_{z \in Z} C(z)$.  By the definition of $Z$ we may suppose $Z = \{c_{a_1}, \dots, c_{a_z} \}$ where $s+1 \leq a_1 < \cdots < a_z \leq p$.  Then, since $a_i \geq s+i$ for each $i \in [z]$, $C(c_{a_i}) \leq C(c_{s+i})$ for each $i \in [z]$.  This means
$$|W|+\sum_{i=1}^p C(c_i) - \sum_{z \in Z} C(z) \geq |W|+\sum_{i=1}^p C(c_i) - \sum_{i=s+1}^{s+z} C(c_i).$$
It follows that $f$ satisfies at least
$$\frac{1}{k-1} \left( |W|+\sum_{i=1}^p C(c_i) - \sum_{i=s+1}^{s+z} C(c_i) \right)$$
vertex requests of $r$.

Construct $g$ as follows. For each $x\in D_{1}$, let $g(x) = r(x)$. For each $x\in X_1- D_1$, let $g(x)$ be an element of $L(x)$.  As before, let $G' = G - X_1$.  Let $L''(v) = L(v) - g(X)$ for each $v \in V(G')$, and let $r''$ be the function $r$ with its domain restricted to $D' - r^{-1}(g(X))$.  Clearly $r''$ is a request of $L''$, $G'= K_{n_2,\dots,n_k}$, and $|L'(v)| \geq s + 1 - n_1 = n_2+\cdots+n_{k-1} + 1$ for each $v \in V(G')$.

So, the induction hypothesis tells us that we can complete $g$ according to a proper $L''$-coloring of $G'$ that satisfies at least  $|D' - r^{-1}(g(X))|/(k-1)$ of the vertex requests of $r''$.  Since
$$\left|D' - r^{-1}(g(X))\right| = |W|+\sum_{i=1}^p C(c_i) - \sum_{c \in g(X)} C(c) \geq |W|+\sum_{i=n_1+1}^p C(c_i),$$
it follows that $g$ satisfies at least
$$|D_1| + \frac{1}{k-1} \left( |W|+\sum_{i=n_1+1}^p C(c_i) \right)$$
vertex requests of $r$.

For the sake of contradiction, assume 
$$\frac{1}{k-1} \left( |W|+\sum_{i=1}^p C(c_i) - \sum_{i=s+1}^{s+z} C(c_i) \right)<|D|/k$$ and  
$$|D_1| + \frac{1}{k-1} \left( |W|+\sum_{i=n_1+1}^p C(c_i) \right) <|D|/k.$$ 
Since $|W|+ |D_1|+ \sum_{i=1}^p C(c_i)= |D|$, rearranging these inequalities gives
$$(k-1)|D_1| > |W| + \sum_{i=1}^p C(c_i) - k \sum_{i=s+1}^{s+z} C(c_i) \geq \sum_{i=1}^p C(c_i) - k \sum_{i=s+1}^{s+z} C(c_i) $$
and
$$(k-1)|D_1| < \frac{-|W|}{k-1} - \frac{k}{k-1}\sum_{i=n_1+1}^p C(c_i)+\sum_{i=1}^p C(c_i) \leq - \frac{k}{k-1}\sum_{i=n_1+1}^p C(c_i)+\sum_{i=1}^p C(c_i).$$
These inequalities imply that 
$$\frac{k}{k-1}\sum_{i=n_1+1}^p C(c_i) -  k \sum_{i=s+1}^{s+z} C(c_i) < 0$$
which implies that $\sum_{i=n_1+1}^p C(c_i) - (k-1)\sum_{i=s+1}^{s+z} C(c_i) < 0$.  Now, for each $j \in [k-1]$, let $s_j = \sum_{t=1}^j n_t$ which means $s_{k-1}=s$, and note that
\begin{align*}
\sum_{i=n_1+1}^p C(c_i) - (k-1)\sum_{i=s+1}^{s+z} C(c_i) &\geq \sum_{i=n_1+1}^{s+z} C(c_i) - (k-1)\sum_{i=s+1}^{s+z} C(c_i) \\
&=\sum_{i=n_1+1}^{s} C(c_i) - (k-2)\sum_{i=s+1}^{s+z} C(c_i) \\
&= \sum_{j=1}^{k-2} \sum_{i=1 + s_j}^{s_{j+1}} C(c_i) - (k-2)\sum_{i=s+1}^{s+z} C(c_i) \\
&= \sum_{j=1}^{k-2} \left(\sum_{i=1 + s_j}^{s_{j+1}} C(c_i) - \sum_{i=s+1}^{s+z} C(c_i) \right). 
\end{align*}
Finally, for each $j \in [k-2]$, note that $\sum_{i=1 + s_j}^{s_{j+1}} C(c_i)$ has $n_{j+1}$ terms while $\sum_{i=s+1}^{s+z} C(c_i)$ has $z$ terms, and we know $z \leq n_1 \leq n_{j+1}$.  Also, $C(c_{n_1+1}) \geq C(c_{n_1+2}) \geq \cdots \geq C(c_{s+z})$.  Consequently,
$$\sum_{j=1}^{k-2} \left(\sum_{i=1 + s_j}^{s_{j+1}} C(c_i) - \sum_{i=s+1}^{s+z} C(c_i) \right) \geq 0$$
which is a contradiction.

Thus, $f$ or $g$ is a proper $L$-coloring of $G$ that satisfies at least $|D|/k$ vertex requests of $r$, and the induction step is complete.  
\end{proof}

\section{The List Flexibility Function} \label{small}

We begin this section by showing that the bound in Theorem~\ref{thm: awesome_sauce} can be strict.

\begin{pro} \label{thm: KNN}
For $G=K_{n,n}$ with $n \geq 3$, we have $\chi_{\ell flex}(G) \leq n < n+1 = \col(G)$.
\end{pro}

It is worth mentioning that for large $n$, one may expect a much stronger upper bound for $\chi_{\ell flex}(K_{n,n})$ to hold.  In fact, it is unknown whether $\chi_{\ell flex}(K_{n,n}) = O(\log_2(n))$ as $n \rightarrow \infty$ (see the discussion surrounding Conjectures 15 and 16 in~\cite{KM23}).

Before we prove Proposition~\ref{thm: KNN}, we need two results.  The first is a well-known list coloring result.  A graph $G$ is said to be \emph{degree-choosable} if a proper $L$-coloring of $G$ exists whenever $L$ is a list assignment for $G$ satisfying $|L(v)| = d(v)$ for each $v \in V(G)$.

\begin{thm} [\cite{B77, ET79}] \label{thm: degree}
Every connected graph $G$ is degree-choosable unless each block of $G$ is a complete graph or an odd cycle.  Consequently, any complete bipartite graph with both partite sets of size at least two is degree-choosable. 
\end{thm}

Now, we prove a lemma related to the degree-choosability of stars.

\begin{lem} \label{lem: P3}
Suppose $G=K_{1,n}$ and the bipartition of $G$ is $X=\{x\}$, $Y=\{y_1,\ldots,y_n\}$. Suppose $L$ is a list assignment for $G$ such that $|L(x)| \geq n$ and $|L(y_i)| \geq 1$ for all $i \in [n]$. Then, $G$ is not $L$-colorable if and only if  $|L(y_i)|=1$ for all $i \in [n]$, $L(y_1),\ldots,L(y_n)$ are pairwise disjoint, and $L(x)=\bigcup_{i \in [n]} L(y_i)$.
\end{lem}
\begin{proof}
The proof is trivial when $n=1$. So, assume $n \geq 2$.  Suppose $G$ is not $L$-colorable.  

For the sake of contradiction, assume that there is $k \in [n]$ such that $|L(y_k)|>1$.  We will obtain a contradiction by constructing a proper $L$-coloring $f$ of $G$.  Color $y_i$ so that $f(y_i) \in L(y_i)$ for all $i \in [n]-\{k\}$. Since $L(x)-\{f(y_i) : i \in [n]-\{k\}\}$ is nonempty, we can color $x$ so that $f(x) \in L(x)-\{f(y_i) : i \in [n]-\{k\}\}$. Since $L(y_k)-\{f(x)\}$ is nonempty, we can complete our proper $L$-coloring by coloring $y_k$ so that $f(y_k) \in L(y_k)-\{f(x)\}$. Having reached a contradiction, we have that $|L(y_i)|=1$ for all $i \in [n]$. 

Now, assume for the sake of contradiction that $L(y_1),\ldots,L(y_n)$ are not pairwise disjoint. Then there are $j,k \in [n]$ such that $j \neq k$ and $L(y_j)=L(y_k)$. We will obtain a contradiction by constructing a proper $L$-coloring $f$ of $G$.  Let $f(y_j)$ and $f(y_k)$ be the element in $L(y_k)$. For each vertex $y \in Y-\{y_j,y_k\}$, let $f(y)$ be the element in $L(y)$. Since $|\{f(y) : y \in Y\}|<n$, we can complete a proper $L$-coloring of $G$ by coloring $x$ with an element in $L(x)-\{f(y) : y \in Y\}$. Having reached a contradiction, we have that $L(y_1),\ldots,L(y_n)$ are pairwise disjoint.

Now, assume for the sake of contradiction that $L(x) \neq \bigcup_{i \in [n]} L(y_i)$.  We will obtain a contradiction by constructing a proper $L$-coloring $f$ of $G$.  Since $L(y_1),\ldots,L(y_n)$ are pairwise disjoint and $|L(y_i)|=1$ for each $i \in [n]$, we may suppose $L(y_i)=\{c_i\}$ so that $\bigcup_{i \in [n]} L(y_i)=\{c_1,\ldots,c_n\}$. Let $f(y_i)=c_i$ for all $i \in [n]$. Since $|L(x)|=n$ and $L(x) \neq \{c_1,\ldots,c_n\}$, note that $L(x)-\{c_1,\ldots,c_n\}$ is nonempty.  We can complete a proper $L$-coloring of $G$ by coloring $x$ with an element of $L(x)-\{c_1,\ldots,c_n\}$.  Having reached a contradiction, we have that $L(x)=\bigcup_{i \in [n]} L(y_i)$.

Conversely, suppose $|L(y_i)|=1$ for all $i \in [n]$, $L(y_1),\ldots,L(y_n)$ are pairwise disjoint, and $L(x)=\bigcup_{i \in [n]} L(y_i)$. Suppose $L(y_i)=\{c_i\}$ as before. Assume for the sake of contradiction that $f$ is a proper $L$-coloring of $G$. Then $f(y_i)=c_i$ for all $i \in [n]$. Since $L(x)=\{c_1,\ldots,c_n\}$, $f(x) \in \{c_1,\ldots,c_n\}$ which contradicts the fact that $f$ is proper.
\end{proof}

We are now ready to prove Proposition~\ref{thm: KNN}.

\begin{proof}
First, note that Statement~(ii) of Theorem~\ref{thm: K3n} implies the desired result when $n=3$.  So, we assume that $n \geq 4$.  Suppose the partite sets of $G$ are $X=\{x_1,\ldots,x_n\}$ and $Y=\{y_1,\ldots,y_n\}$. Suppose $L$ is an arbitrary $n$-assignment for $G$, and $r:D \rightarrow \bigcup_{v \in V(G)}L(v)$ is a request of $L$. Since $\rho(G)=2$, if we can show there is a proper $L$-coloring of $G$ that satisfies at least $|D|/2$ vertex requests, we will have $\chi_{\ell flex}(G) \leq n$.  We will show such a coloring exists in each of the following cases: (1) $|D| \in \{2n-1,2n\}$, (2) $|D| \in \{2n-3,2n-2\}$, and (3) $|D| \leq 2n-4$.

First, suppose that $|D| \in \{2n-1,2n\}$. Without loss of generality, assume $X \subseteq D$ and $\{y_1,\ldots,y_{n-1}\} \subseteq D$. Suppose $r(x_i) = c_i$ for each $i \in [n]$. We first show that we may assume $c_1,\ldots,c_n$ are pairwise distinct. Suppose there are $i,j \in [n]$ with $i \neq j$ such that $c_i=c_j$.  We can construct a proper $L$-coloring $f$ of $G$ that satisfies $n \geq |D|/2$ vertex requests as follows.  Let $f(x_i)=c_i$ for all $i \in [n]$ and color each $y \in Y$ with an element in  $L(y)-\{c_1,\ldots,c_n\}$.  So, we may assume that $c_1,\ldots,c_n$ are pairwise distinct. 

Now, we construct a proper $L$-coloring $f$ of $G$ as follows.  Let $f(y_1)=r(y_1)$. Without loss of generality, assume that if $r(y_1) \in \{c_1,\ldots,c_n\}$, then $r(y_1)=c_n$.  Now, let $f(x_i)=c_i$ for all $i \in [n-1]$.  By Lemma~\ref{lem: P3}, we can complete a proper $L$-coloring of $G$ that satisfies $n$ vertex requests unless there are pairwise distinct colors $d_2, \ldots, d_n$ such that $L(y_i)=\{c_1,\ldots,c_{n-1},d_i\}$ for each $i \in \{2, \ldots, n\}$ and $L(x_n)=\{r(y_1),d_2,\ldots,d_n\}$ (In the case $r(y_1) \neq c_n$, $c_n \in \{d_2, \ldots, d_n\}$ and $r(y_1) \notin \{d_2, \ldots, d_n\}$, and in the case $r(y_1) = c_n$, $c_n \notin \{d_2, \ldots, d_n\}$).  So, assume these properties hold, and we will reconstruct the coloring $f$. 

If there exists a $k \in \{2, \ldots, n\}$ such that $r(y_k) = c_q$ for some $q \in [n-1]$, then let $f(y_i)=c_q$ for each $i \in \{2, \ldots, n\}$. Let $f(x_i)=c_i$ for each $i \in [n] - \{q\}$ and color $y_1$ so that $f(y_1) \in L(y_1)-\{c_i : i \in [n] - \{q\} \}$. Finally, color $x_1$ so that $f(x_1) \in L(x_1)-\{c_q,f(y_1)\}$.  It is easy to see that $f$ is a proper $L$-coloring that satisfies $n$ vertex requests.  If $r(y_i) \notin \{c_1,\ldots,c_{n-1}\}$ for each $i \in \{2, \ldots, n\}$, then $r(y_i)=d_i$ for each $i \in \{2, \ldots, n-1\}$ . Let $f(y_2)=d_2$, $f(x_i)=c_i$ for each $i \in \{2, \ldots, n\}$, and $f(y_i)=c_1$ for each $i \in \{3, \ldots, n\}$.  Then, let $f(y_1)$ be an element in $L(y_1)-\{c_2,\ldots,c_n\}$. Finally, we can complete a proper $L$-coloring that satisfies at least $n$ vertex requests by letting $f(x_1)$ be an element in $L(x_1)-\{f(y_1),d_2,c_1\}$ (which is possible since $n \geq 4$).

Second, suppose $|D| \in \{2n-3,2n-2\}$.  Since $n \geq 4$, we may assume without loss of generality that $\{x_1,\ldots,x_{n-1}\} \subseteq D$ and $y_1 \in D$.  Suppose $r(x_i) = c_i$ for each $i \in [n]$.  We will now construct a proper $L$-coloring $f$ of $G$ that satisfies at least $n-1 \geq |D|/2$ vertex requests.  Let $f(x_i)=c_i$ for $1 \leq i \leq n-1$.  By Lemma~\ref{lem: P3}, we can complete a proper $L$-coloring of $G$ that satisfies $n-1$ vertex requests unless $c_1, \ldots, c_{n-1}$ are pairwise distinct, there are pairwsie distinct colors $d_1, \ldots, d_n$ such that $L(y_i)=\{c_1,\ldots,c_{n-1},d_i\}$ for each $i \in [n]$, and $L(x_n)=\{d_1,d_2,\ldots,d_n\}$.  So, assume these properties hold, and we will reconstruct the coloring $f$.

Begin by letting $f(y_1) = r(y_1)$.  Then, define $f$ so that $n-2$ of the vertices in $\{x_i \in \{x_1,\ldots,x_{n-1}\} : r(x_i) \neq r(y_1) \}$ are colored with their requested color (this is possible since at most one vertex in $\{x_1,\ldots,x_{n-1}\}$ requests $r(y_1)$).  Call the set of vertices that have been colored at this stage $A$, and let $B = A - \{y_1\}$.  Then, let $G' = G - A$, and notice that $G'= K_{2,n-1}$.  Let $L'$ be the list assignment for $G'$ given by $L'(x) = L(x) - \{r(y_1)\}$ for each $x \in (X - B)$ and $L'(y) = L(y) - \{f(x) : x \in B\}$ for each $y \in Y - \{y_1\}$.  Clearly, $|L'(x)| \geq n-1 = d_{G'}(x)$ for each $x \in (X - B)$ and $|L'(y)| \geq 2 = d_{G'}(y)$ for each $y \in Y - \{y_1\}$.  So, Theorem~\ref{thm: degree} implies there is a proper $L'$-coloring of $G'$ which means that we can extend $f$ to a proper $L$-coloring of $G$ that satisfies at least $n-1$ vertex requests.

Finally, suppose that $|D| \leq 2n-4$.  Assume without loss of generality that $|D \cap X| \geq |D|/2$, and let $A$ be a set consisting of $\lceil |D|/2 \rceil$ arbitrarily chosen vertices of  $D \cap X$.  Now, color each element of $A$ with its requested color.  Then, let $G' = G - A$, and notice that $G'= K_{n-|A|,n}$ and $n-|A| \geq n - (2n-4)/2 = 2$.  Let $L'$ be the list assignment for $G'$ given by $L'(x) = L(x)$ for each $x \in (X - A)$ and $L'(y) = L(y) - \{r(x) : x \in A\}$ for each $y \in Y$.  Clearly, $|L'(x)| = n = d_{G'}(x)$ for each $x \in X$ and $|L'(y)| \geq n-|A| = d_{G'}(y)$ for each $y \in Y$.  So, Theorem~\ref{thm: degree} implies there is a proper $L'$-coloring of $G'$ which means that we can extend our coloring to a proper $L$-coloring of $G$ that satisfies at least $|D|/2$ vertex requests.      
\end{proof}

Recall that $\chi_{\ell}(K_{2,n})=2$ when $n \in [3]$, and $\chi_{\ell}(K_{2,n})=3$ otherwise.  Having proved Theorem~\ref{thm: awesome_sauce}, we are ready to prove Theorem~\ref{cor: K2n} which we restate.

\begin{customthm} {\bf \ref{cor: K2n}}
Suppose $G = K_{2,n}$
\[\epsilon_{\ell}(G,t)=
\begin{cases}
1/2 \text{ if } n=1 \text{ and } t\geq2
\\0 \text{ if } n\in \{2,3\} \text{ and } t=2 
\\1/2 \text{ otherwise}
\end{cases}
.\]
\end{customthm}

\begin{proof}
If $n=1$, we know that $\chi_{\ell flex}(G) = 2$. Thus when $n = 1$, $\epsilon_{\ell}(G,t) = 1/2$ for $t \geq 2$. Also if $n \geq 2$ and $t \geq 3$, $\epsilon_{\ell}(G,t) = 1/2$ by Theorem~\ref{thm: awesome_sauce}.

Suppose $n\in \{2,3\}$. Suppose the partite sets of $G$ are $X =\{x_1, x_2\}$ and $Y =\{y_1, \ldots, y_n\}$.  Let $L$ be the 2-assignment for $G$ given as follows.  Let $L(x_i) = \{2i-1, 2i\}$ for each $i \in [2]$. Let $L(y_1) = \{1,3\}$, $L(y_2) = \{1,4\}$, and if $n = 3$ let $L(y_3) = [2]$. Suppose $r$ is the request of $L$ with domain $\{x_1\}$ given by $r(x_1)=1$.  Suppose for the sake of contradiction there is a proper $L$-coloring of $G$, $f$, such that $f(x_1)=1$.  Notice $f(y_1)=3$ and $f(y_2)=4$. Consequently, $f(x_2)=f(y_1)$ or $f(x_2)=f(y_2)$. This means $f$ is not proper which is a contradiction. Thus  $\epsilon_{\ell}(G,2) = 0$.
\end{proof}

With Theorem~\ref{cor: K2n} in mind, we now turn our attention to the list flexibility function of $K_{3,n}$ for $n \geq 3$ and determining the value of $\epsilon_{\ell}(G,3)$ when $G$ is a $3$-choosable complete bipartite graph.  As was mentioned in Section~\ref{intro}, it is known that for $3 \leq m \leq n$, $K_{m,n}$ is 3-choosable if and only if $m=3$ and $n \leq 26$, $m=4$ and $n \leq 20$, $m=5$ and $n \leq 12$, or $m=6$ and $n \leq 10$.  It is also easy to show that $\chi_{\ell}(K_{3,n}) \geq 3$ whenever $n \geq 3$. 

We begin by proving results on the $(a,b)$-choosability of complete bipartite graphs that will be used later.  Our first such result is a generalization of Proposition 5 in~\cite{AC21}.

\begin{lem} \label{lem: precolor}
Suppose $G = K_{n, b}$.  Then, $G$ is $(t,n)$-choosable if and only if $b < t^n$.
\end{lem}

\begin{proof}

When $b \geq t^n$, the fact that $G$ is not $(t,n)$-choosable follows from Proposition~5 in~\cite{AC21}.  Conversely suppose $b < t^n$, and $L$ is a $(t,n)$-assignment for $G$.  Suppose the partite sets of $G$ are $X= \{x_1, \ldots, x_n\}$, $Y= \{y_1, \ldots, y_b \}$.  We will construct a proper $L$-coloring, $f$, of $G$ in each of two cases.

Suppose there are $i,j \in [n]$ with $i \neq j$ such that there is a $c \in L(x_i) \cap L(x_j)$.  Let $f(x_i) = f(x_j) = c$.  Then, for each $v \in X -\{x_i,x_j\}$ let $f(v)$ be some element in $L(v)$.  Notice $|f(X)| \leq n-1$.  So, we can complete a proper $L$-coloring of $G$ by letting $f(v)$ be some element in $L(v) - f(X)$ for each $v \in Y$.

Now, suppose $L(x_1), \ldots, L(x_n)$ are pairwise disjoint. Let
$$A= \{\{a_1, \ldots, a_n\} : a_i \in L(x_i) \text{ for each } i \in [n]\},$$
and notice $|A| = t^n$.  So, we may suppose there is a $\{c_1, \ldots, c_n \} \in A - \{L(y_i): i \in [b]\}$ where $c_i \in L(x_i)$ for each $i \in [n]$.   Now, let $f(x_i) = c_i$ for each $i \in [n]$.    Notice $f(X) \neq L(y_j)$ for each $j \in [b]$.  So, we can complete a proper $L$-coloring of $G$ by letting $f(v)$ be some element in $L(v) - f(X)$ for each $v \in Y$.
\end{proof}

We now give a complete characterization of $(3,2)$-choosable complete bipartite graphs to help with the proof of Theorem~\ref{thm: K3n}.  Before beginning the proof, we mention an observation and two lemmas that we will need.

\begin{obs} \label{obs: P4}
Suppose $H$ is a subgraph of $G$, and $g:V(G) \rightarrow \mathbb{N}$ is a function. Suppose $h$ is the restriction of $g$ to $V(H)$. If $G$ is $g$-choosable, then $H$ is $h$-choosable.
\end{obs}

\begin{lem} [\cite{CR12}]\label{lem: smallpot}
Suppose that $G$ is a graph and $f: V(G) \rightarrow \mathbb{N}$ is a function satisfying $f(v) < |V(G)|$ for all $v \in V(G)$.  If $G$ is not $f$-choosable, then there is an $f$-assignment $L$ for $G$ satisfying $|\bigcup_{v \in V(G)} L(v)| < |V(G)|$ such that there is no proper $L$-coloring of $G$.
\end{lem}

\begin{lem}\label{lem: k_33}
  Suppose $G=K_{3,3}$, and $f: V(G) \rightarrow \{2,3\}$ is such that there is a unique $x \in V(G)$ with $f(x)=3$.  Then $G$ is $f$-choosable.  
\end{lem}
\begin{proof}  Suppose the bipartition of $G$ is $X = \{x_1, x_2, x_3 \}$, $Y = \{y_1, y_2, y_3 \}$ throughout the proof. Suppose without loss of generality that $f(x_3)=3$.  Suppose $L$ is an arbitrary list assignment for $G$ such that $|L(v)|=f(v)$ for each $v \in V(G)$. To prove the desired result, we will construct a proper $L$-coloring of $G$.  Notice that in the case $\bigcap_{i=1}^3 L(y_i) \neq \emptyset$, we can greedily construct a proper $L$-coloring by coloring the vertices in $Y$ with an element from $\bigcap_{i=1}^3 L(y_i)$, and then greedily coloring the vertices in $X$.  

Now, we show a proper $L$-coloring of $G$ exists when there is a $c \in L(x_1) \cap L(x_2)$.  In this case we may assume $c \notin L(x_3)$, and $\{L(y_i):i \in [3]\} = \{\{c,z\} : z \in L(x_3) \}$ since otherwise we could construct a proper $L$-coloring of $G$ where $x_1$ and $x_2$ are colored with $c$.  This however implies that $c \in \bigcap_{i=1}^3 L(y_i)$ which implies a proper $L$-coloring of $G$ must exist. 

Finally, suppose $L(x_1) \cap L(x_2) = \emptyset$. 
 Let $L'$ be the 2-assignment for $G'=G - x_3$ obtained by restricting the domain of $L$ to $V(G')$.  Since $G'$ is 2-choosable there is a proper $L'$-coloring of $G'$ which we will call $g$.  We may suppose $L(x_3) = \{g(y_i) : i \in [3]\}$ since otherwise we could extend $g$ to a proper $L$-coloring of $G$.  For each $i \in [2]$ and $j \in [3]$, let $a_i$ be the element in the set $L(x_i) - \{g(x_i)\}$ and let $b_j$ be the element in the set $L(y_j) - \{g(y_j)\}$.  For each $j \in [3]$, we may assume $b_j \in \{g(x_1),g(x_2)\}$ since otherwise we could obtain a proper $L$-coloring of $G$ by recoloring $y_j$ with $b_j$ and then coloring $x_3$ with $g(y_j)$. Thus, $\{b_1, b_2, b_3\}\subseteq \{g(x_1), g(x_2)\}$. Since $L(x_1) \cap L(x_2) = \emptyset$, $\{a_1, a_2\} \cap \{g(x_1), g(x_2)\}=\emptyset$.  Consider the $L'$-coloring of $G'$ given by $h(x_i) = a_i$ for each $i \in [2]$ and $h(y_j)=b_j$ for each $j \in [3]$.  It is easy to see that $h$ is proper.  Since $|\{b_1,b_2,b_3\}| \leq 2$, we can extend $h$ to a proper $L$-coloring of $G$ by coloring $x_3$ with an element of $L(x_3) - \{b_1,b_2,b_3\}$. 
\end{proof}

\begin{customthm} {\bf \ref{thm: 3_2}}
Suppose $G=K_{m,n}$ where $m,n \in \mathbb{N}$ satisfy $m \leq n$.
\begin{enumerate}[label=(\roman*)]
    \item We have $G$ is $(3,2)$-choosable if and only if $m=1$ and $n \in \mathbb{N}$, $m=2$ and $n \leq 8$, $m=3$ and $n \leq 6$, or $m=4$ and $n=4$. 
    \item We have $G$ is $(2,3)$-choosable if and only if $m \in [2]$ and $n \in \mathbb{N}$, $m=3$ and $n \leq 7$, or $m=4$ and $n \leq 5$.
\end{enumerate}
\end{customthm}

\begin{proof} Suppose the bipartition of $G$ is $X = \{x_1, \ldots, x_m \}$, $Y = \{y_1, \ldots, y_n \}$ throughout the proof. Proving Statements~(i) and~(ii) require us to prove four implications which we will write as four claims.  We begin by proving two claims that imply Statement~(i) 
\\
\\
\noindent \textbf{Claim A.} $G$ is not $(3,2)$-choosable when $m=2$ and $n \geq 9$, $m=3$ and $n \geq 7$, $m=4$ and $n \geq 5$, or $m \geq 5$. 
\\
\\
If $m=2$ and $n \geq 9$, then $G$ is not $(3,2)$-choosable by Lemma~\ref{lem: precolor}. If $m=3$ and $n \geq 7$, it is sufficient to show that $G$ is not $(3,2)$-choosable when $n=7$ by Observation~\ref{obs: P4}. For $G=K_{3,7}$, let $L$ be the $(3,2)$-assignment for $G$ defined by $L(x_1)=\{1,2,3\}$, $L(x_2)=\{1,3,4\}$, $L(x_3)=\{2,4,5\}$, $L(y_1)=\{1,2\}$, $L(y_2)=\{1,4\}$, $L(y_3)=\{1,5\}$, $L(y_4)=\{2,3\}$, $L(y_5)=\{2,4\}$, $L(y_6)=\{3,4\}$, and $L(y_7)=\{3,5\}$. Assume for the sake of contradiction that $f$ is a proper $L$-coloring of $G$. If $f(x_1)=1$, then $f(y_1)=2$, $f(y_2)=4$, and $f(y_3)=5$. Then $f(x_3) \in L(x_3)-\{2,4,5\}$ which is a contradiction. If $f(x_1)=2$, then $f(y_1)=1$, $f(y_4)=3$, and $f(y_5)=4$. Then $f(x_2) \in L(x_2)-\{1,3,4\}$ which is a contradiction. Finally, if $f(x_1)=3$, we have $f(y_4)=2$, $f(y_6)=4$, and $f(y_7)=5$. Then $f(x_3) \in L(x_3)-\{2,4,5\}$ which is a contradiction. So, $f(x_1) \notin L(x_1)$ which is a contradiction. Therefore, $G$ is not $L$-colorable. 

If $m=4$ and $n \geq 5$, it is sufficient to show that $G$ is not $(3,2)$-choosable when $n=5$ by Observation~\ref{obs: P4}. For $G=K_{4,5}$, let $L$ be the $(3,2)$-assignment for $G$ defined by  $L(x_1)=\{1,2,5\}$, $L(x_2)=\{1,2,6\}$, $L(x_3)=\{3,4,5\}$, $L(x_4)=\{3,4,6\}$, $L(y_1)=\{1,3\}$, $L(y_2)=\{1,4\}$, $L(y_3)=\{2,3\}$, $L(y_4)=\{2,4\}$, $L(y_5)=\{5,6\}$. Assume for the sake of contradiction that $f$ is a proper $L$-coloring of $G$. If $f(x_1)=1$, then $f(y_1)=3$ and $f(y_2)=4$. Thus, $f(x_3)=5$ which implies $f(y_5)=6$. So, $f(x_4) \in L(x_4)-\{3,4,6\}$ which is a contradiction. If $f(x_1)=2$, then $f(y_3)=3$ and $f(y_4)=4$. Thus, $f(x_3)=5$ which means $f(y_5)=6$. Therefore, $f(x_4) \in L(x_4)-\{3,4,6\}$ which is a contradiction. Finally, if $f(x_1)=5$, then $f(y_5)=6$. This leaves us with $f(x_4) \in \{3,4\}$. If $f(x_4)=3$, then $f(y_1)=1$ and $f(y_3)=2$. Therefore, $f(x_2) \in L(x_2)-\{1,2,6\}$ which is a contradiction. If $f(x_4)=4$, then $f(y_2)=1$ and $f(y_4)=2$. This means $f(x_2) \in L(x_2)-\{1,2,6\}$, which is another contradiction. Therefore, $G$ is not $L$-colorable. If $m \geq 5$, then $G$ contains a copy of $K_{4,5}$ as a subgraph. So $G$ is not $(3,2)$-choosable by Observation~\ref{obs: P4}.  This completes the proof of Claim A.
\\
\\
\noindent \textbf{Claim B.}  If $m=1$ and $n \in \mathbb{N}$, $m=2$ and $n \leq 8$, $m=3$ and $n \leq 6$, or $m=4$ and $n=4$, then $G$ is $(3,2)$-choosable
\\
\\
If $m=1$ and $n \in \mathbb{N}$, $G$ is a tree. Since $G$ is a tree it is 2-choosable which means it is also $(3,2)$-choosable.  If $m=2$ and $n \leq 8$, then $G$ is $(3,2)$-choosable by Lemma~\ref{lem: precolor}. 

Suppose $m=3$ and $n \leq 6$.  We must show that $G$ is $(3,2)$-choosable.  By Observation~\ref{obs: P4} we may assume that $n=6$.  For the sake of contradiction suppose $G$ is not $(3,2)$-choosable. Then by Lemma~\ref{lem: smallpot}, there is a $(3,2)$-assignment $L$ for $G$ satisfying $|\bigcup_{v \in V(G)} L(v)| < 9$ such that there is no proper $L$-coloring of $G$. Let $\mu : \bigcup_{i = 1}^{6} L(y_i) \rightarrow \mathbb{N}$ be the function given by $\mu(c) = |\{i\in [n]:c\in L(y_i)\}|$.  Let $P=\bigcup_{i=1}^n L(y_i)$ and $M = \max\limits_{c\in P} \mu(c)$. If $M=1$, then $|P|=12$ which implies $|\bigcup_{v \in V(G)} L(v)| \geq 12$ which is a contradiction.  Suppose $M=2$.  We will obtain a contradiction by constructing a proper $L$-coloring of $G$.  Since $|\bigcup_{v \in V(G)} L(v)| < 9$, we have that $L(x_1)$, $L(x_2)$, and $L(x_3)$ are not pairwise disjoint. So, we may assume without loss of generality that $c \in L(x_1) \cap L(x_2)$.  Color $x_1$ and $x_2$ with $c$.  Since $M=2$, we know $c$ is in at most two of the lists $L(y_1), \ldots, L(y_6)$.  For any vertex $y \in Y$ whose list contains $c$, color $y$ with the color in $L(y)- \{c\}$.  Finally we can greedily complete a proper $L$-coloring by coloring $x_3$ followed by any vertices in $Y$ that have not yet been colored. 

Now, suppose that $M=3$ and $\mu(c_1)=3$.  Note that if $c_1 \in \bigcap_{i=1}^3 L(x_i)$, we can construct a proper $L$-coloring by coloring all the vertices in $X$ with $c_1$, and then greedily coloring the vertices in $Y$.  So, we may assume $c_1 \notin \bigcap_{i=1}^3 L(x_i)$ and without loss of generality assume $c_1\notin L(x_3)$.  We will now show that a proper $L$-coloring of $G$ exists.  First color the three vertices in $Y$ whose list contains $c_1$ with $c_1$. Assume without loss of generality these vertices are $y_1,y_2,y_3$.  Let $L'$ be the list assignment for $G - \{y_1,y_2,y_3\}$ defined as follows: $L'(y_j) = L(y_j)$ for each $j \in \{4,5,6\}$, $L'(x_3)=L(x_3)$, then for each $i \in [2]$ obtain $L'(x_i)$ from $L(x_i)$ by deleting $c_1$ from $L(x_i)$ if $c_1 \in L(x_i)$ or arbitrarily deleting a color from $L(x_i)$ in the case $c_1 \notin L(x_i)$.  By Lemma~\ref{lem: k_33} a proper $L'$-coloring of $G - \{y_1,y_2,y_3\}$ exists and such a coloring can be used to complete a proper $L$-coloring of $G$. Finally, suppose that $4 \leq M \leq 6$ and $\mu(c_1) \geq 4$.  Now, color all of the vertices in $Y$ whose list contains $c_1$ with $c_1$. After this let $C$ be the set of vertices in $Y$ that have been colored.  Now let $L'$ be the list assignment for $G-C$ given by $L'(v) = L(v) - \{c_1\}$ for each $v \in V(G-C)$.  Since $|L'(v)| \geq 2$ for each $v \in V(G-C)$ and $G-C$ is a subgraph of a 2-choosable graph (namely, $K_{2,3}$), a proper $L'$-coloring of $G - C$ exists and such a coloring can be used to complete a proper $L$-coloring of $G$. 

Now we turn our attention to the case where $m=n=4$.  We must show that $G$ is $(3,2)$-choosable.  For the sake of contradiction suppose $G$ is not $(3,2)$-choosable.  Then by Lemma~\ref{lem: smallpot}, there is a $(3,2)$-assignment $L$ for $G$ satisfying $|\bigcup_{v \in V(G)} L(v)| < 8$ such that there is no proper $L$-coloring of $G$. Let $\mu : \bigcup_{i = 1}^{4} L(y_i) \rightarrow \mathbb{N}$ be the function given by $\mu(c) = |\{i\in [n]:c\in L(y_i)\}|$.  Let $P=\bigcup_{i=1}^n L(y_i)$ and $M = \max\limits_{c\in P} \mu(c)$. If $M=1$, then $|P|=8$ which implies $|\bigcup_{v \in V(G)} L(v)| \geq 8$ which is a contradiction. Now, suppose that $M=2$ and $\mu(c_1)=2$.  Note that if $c_1 \in \bigcap_{i=1}^4 L(x_i)$, we can construct a proper $L$-coloring by coloring all the vertices in $X$ with $c_1$, and then greedily coloring the vertices in $Y$.  So, we may assume $c_1 \notin \bigcap_{i=1}^4 L(x_i)$. We will now show that a proper $L$-coloring of $G$ exists.  First color the two vertices in $Y$ whose list contains $c_1$ with $c_1$. Assume without loss of generality these vertices are $y_1$ and $y_2$. If there is a $c_2\in L(y_3)\cap L(y_4)$, we complete a proper $L$-coloring by coloring $y_3$ and $y_4$ with $c_2$ and greedily coloring the vertices in $X$.    So, we may assume $L(y_3)\cap L(y_4)=\emptyset$. Now let $L'$ be the list assignment for $G-\{y_1,y_2\}$ given by $L'(v) = L(v) - \{c_1\}$ for each $v \in V(G-\{y_1,y_2\})$.  Notice $L(y_i) = L'(y_i)$ for $i \in \{3,4\}$.  Let $A=\{\{a,b\}: a\in L'(y_3), b\in L'(y_4)\}$ and $B= \{L'(x_i): i\in [4]\}$. Note that $A$ contains four sets each of size two. Also, note each element of $B$ has size at least two, and $B$ has at least one element of size three. Thus, there is a $\{c,d\}\in A-B$, and we assume $c\in L'(y_3)$ and $d\in L'(y_4)$. Notice that $L(x_i)-\{c,d\} \neq \emptyset$ for each $i\in [4]$.  So, we can complete a proper $L$-coloring of $G$ by coloring $y_3$ with $c$, $y_4$ with $d$, and $x_i$ with an element of $L(x_i) - \{c,d\}$ for each $i \in [4]$. 

Finally, suppose that $3 \leq M \leq 4$.  Suppose $\mu(c_1) \geq 3$.  Now, color all of the vertices in $Y$ whose list contains $c_1$ with $c_1$. After this let $C$ be the set of vertices in $Y$ that have been colored.  Let $L'$ be the list assignment for $G-C$ given by $L'(v) = L(v) - \{c_1\}$ for each $v \in V(G-C)$.  Since $|L'(v)| \geq 2$ for each $v \in V(G-C)$ and $G-C$ is a subgraph of a 2-choosable graph (namely, $K_{1,4}$), a proper $L'$-coloring of $G - C$ exists and such a coloring can be used to complete a proper $L$-coloring of $G$.  This completes the proof of Claim B. \\

Next, we prove two implications that imply Statement (ii). 
\\
\\
\noindent \textbf{Claim C.} $G$ is not $(2,3)$-choosable when $m=3$ and $n \geq 8$, $m=4$ and $n \geq 6$, or $m \geq 5$.
\\
\\
If $m=3$ and $n \geq 8$, then $G$ is not $(2,3)$-choosable by Lemma~\ref{lem: precolor}. Also, since $K_{5,5}$ is not (3,2)-choosable, $G$ is not (2,3) choosable whenever $m \geq 5$.

If $m=4$ and $n \geq 6$, it is sufficient to show that $G$ is not $(2,3)$-choosable when $n=6$ by Observation~\ref{obs: P4}. For $G=K_{4,6}$, let $L$ be the $(2,3)$-assignment for $G$ defined by $L(x_1)=\{1,2\}$, $L(x_2)=\{1,3\}$, $L(x_3)=\{4,5\}$, $L(x_4)=\{6,7\}$ $L(y_1)=\{1,4,6\}$, $L(y_2)=\{1,4,7\}$, $L(y_3)=\{1,5,6\}$, $L(y_4)=\{1,5,7\}$, $L(y_5)=\{2,3,4\}$, and $L(y_6)=\{2,3,5\}$. Assume for the sake of contradiction that $f$ is a proper $L$-coloring of $G$.  Note $\{f(x_3), f(x_4)\} \in \{\{4,6\},\{4,7\}, \{5,6\}, \{5,7\} \}$. If $f(x_1)=1$ or $f(x_2)=1$, then there is an $i \in [4]$ such that $L(y_i) \subseteq \{f(x_j): j\in [4]\}$ which contradicts the fact that $f$ is proper. So we may assume $f(x_1)=2$ and $f(x_2)=3$.  This however implies that $L(y_j) = \{f(x_1), f(x_2), f(x_3)\}$ for some $j \in \{5,6\}$ which is a contradiction. Therefore, $G$ is not $L$-colorable. This completes the proof of Claim C.
\\
\\
\noindent \textbf{Claim D.}  If $m \in [2]$ and $n \in \mathbb{N}$, $m=3$ and $n \leq 7$, or $m=4$ and $n \leq 5$, then $G$ is $(2,3)$-choosable.
\\
\\
If $m=1$ and $n \in \mathbb{N}$, $G$ is a tree. Since $G$ is a tree it is 2-choosable which means it is also $(2,3)$-choosable. Also, if $m=3$ and $n\leq 7$, $G$ is (2,3)-choosable by Lemma~\ref{lem: precolor}.  Suppose $m=2$ and $n \in \mathbb{N}$.  We must show that $G$ is $(2,3)$-choosable. Now, suppose $L$ is an arbitrary (2,3)-assignment for $G$. We construct a proper $L$-coloring for $G$ as follows. For each $i\in [2]$, color each $x_i$ with some $c_i \in L(x_i)$. Then greedily color each element in $Y$. Thus, $G$ is $(2,3)$-choosable. 

Finally, suppose $m=4$ and $n \leq 5$. We must show that $G$ is $(2,3)$-choosable. By Observation~\ref{obs: P4} we may assume that $n=5$.  For the sake of contradiction, suppose $G$ is not $(2,3)$-choosable. Then there is a $(2,3)$-assignment $L$ for $G$ such that there is no proper $L$-coloring of $G$. Let $\mu : \bigcup_{i = 1}^{4} L(x_i) \rightarrow \mathbb{N}$ be the function given by $\mu(c) = |\{i\in [4]:c\in L(x_i)\}|$.  Let $P=\bigcup_{i=1}^4 L(x_i)$ and $M = \max\limits_{c\in P} \mu(c)$. If $M=1$, suppose without loss of generality $L(x_i)=\{2i-1,2i\}$ for $i\in [4]$. Let $T=\{L(y_1),L(y_2),L(y_3),L(y_4),L(y_5)\}$ and $C=\{\{c_1,c_2,c_3,c_4\}: c_i \in L(x_i) \text{ for each } i\in [4]$\}. Let $\mathcal{P}$ be the power set of $C$. Then, let $g: T \rightarrow \mathcal{P}$ be the function given by $g(L(y_i))=A$ where $A=\{B\in C:L(y_i)\subseteq B\}$. Since there is no proper $L$-coloring, $\bigcup_{i=1}^5 g(L(y_i))=C$. We claim $|g(L(y_i))|\leq 2$ for each $i\in [5]$. To see why, notice if $i$, $j$, and $k$ are pairwise distinct elements of $[8]$, then at most two elements of $C$ contain $i$, $j$, and $k$.  Consequently, $|g(L(y_i))| \leq 2$ for each $i \in [5]$.  Thus,
$$16 = |C| = \left | \bigcup_{i=1}^5 g(L(y_i)) \right| \leq 10$$
which is a contradiction. 

Now, suppose $M=2$. Without loss of generality, suppose $c_1\in L(x_1)\cap L(x_2)$. If there is a $c_2 \in L(x_3)\cap L(x_4)$, we can complete a proper $L$-coloring of $G$ by coloring $x_1$ and $x_2$ with $c_1$, coloring $x_3$ and $x_4$ with $c_2$, and greedily coloring the remaining vertices.  So, $L(x_3) \cap L(x_4)=\emptyset$. Suppose $L(x_3) = \{a,b\}$ and $L(x_4) = \{c,d\}$.  Notice that it is possible to color each vertex in $X$ with a color from its list and only use: $c_1$, a color from $\{a,b\}$, and a color from $\{c,d\}$. So, we may suppose without loss of generality $L(y_1)=\{c_1,a,c\}$, $L(y_2)=\{c_1,a,d\}$, $L(y_3)=\{c_1,b,c\}$, and $L(y_4)=\{c_1,b,d\}$. Let $d_i$ be the only color in $L(x_i) - \{c_1\}$ for each $i \in [2]$.  Then, suppose $L'(y_5) = L(y_5) - \{d_1, d_2\}$, and notice $|L'(y_5)| \geq 1$.  We can now construct a proper $L$-coloring of $G$ as follows.  First, color $y_i$ with $c_1$ for each $i \in [4]$.  Then, color $y_5$ with a color $c_2 \in L'(y_5)$.  Finally, since $L(x_i) - \{c_1, c_2\}$ is nonempty for each $i \in [4]$, we can greedily complete our proper $L$-coloring. 

Finally, suppose that $3 \leq  M \leq 4$ and $\mu(c_1) \geq 3$.  Now, color all of the vertices in $X$ whose list contains $c_1$ with $c_1$. After this let $C$ be the set of vertices in $X$ that have been colored.  Now let $L'$ be the list assignment for $G-C$ given by $L'(v) = L(v) - \{c_1\}$ for each $v \in V(G-C)$.  Since $|L'(v)| \geq 2$ for each $v \in V(G-C)$ and $G-C$ is a subgraph of a 2-choosable graph (namely, $K_{1,5}$), a proper $L'$-coloring of $G - C$ exists and such a coloring can be used to complete a proper $L$-coloring of $G$. This completes the proof of Claim D.
\end{proof}

We now present a lemma that will be used in combination with Theorem~\ref{thm: 3_2} in order to prove Theorem~\ref{thm: K3n} and Proposition~\ref{pro: forfree}.

\begin{lem} \label{lem: connect}
Suppose $G = K_{m,n}$, $m,n \in \N$, $k \geq 2$, and $G$ is $k$-choosable. If $K_{m,n-1}$ is $(k-1,k)$-choosable and $K_{m-1,n}$ is $(k,k-1)$-choosable then $1/(m+n) \leq \epsilon_{\ell}(G,k)$. If $K_{m,n-1}$ is not $(k-1,k)$-choosable or $K_{m-1,n}$ is not $(k,k-1)$-choosable then $\epsilon_{\ell}(G,k) = 0$.
\end{lem}
\begin{proof}
Let the bipartition of $G$ be $X=\{x_1,\ldots,x_m\}$ and $Y=\{y_1,\ldots,y_n\}$. Suppose $K_{m,n-1}$ is $(k-1,k)$-choosable and $K_{m-1,n}$ is $(k,k-1)$-choosable. Suppose $L$ is an arbitrary $k$-assignment for $G$, $r$ is a request of $L$ with domain $D$, and $D\cap X\neq \emptyset$. Without loss of generality, suppose $x_1 \in D \cap X$.  Let $G' = G - x_1$, $L'(v) = L(v) - r(x_1)$ if $v\in Y$, and $L'(v) = L(v)$ if $v\in X - \{x_1\}$.  Since $G' = K_{m-1,n}$, there is a proper $L'$-coloring of $G'$.  This coloring can be extended to a proper $L$-coloring of $G$ by coloring $x_1$ with $r(x_1)$. Moreover, this proper $L$-coloring satisfies at least $1$ of the vertex requests.  

If $D\cap X = \emptyset$, then $D\cap Y \neq \emptyset$. In this case, a similar argument that uses the fact that $K_{m,n-1}$ is $(k-1,k)$-choosable can be used to show there is a proper $L$-coloring of $G$ that satisfies at least $1$ of the vertex requests. Thus, $\epsilon_{\ell}(G,k) \geq 1/(m+n)$.

Now suppose $K_{m,n-1}$ is not $(k-1,k)$-choosable. Let $G'=G-y_1$. Suppose $L'$ is a $(k-1,k)$-assignment for $G'$ such that there is no proper $L'$-coloring of $G'$. Suppose $c\notin \bigcup_{v\in V(G')} L'(v)$. Let $L$ be the $k$-assignment for $G$ defined as follows \[L(v)= \begin{cases}
L'(v)\cup \{c\} & \text{if } v\in X, \\
L'(x_1)\cup \{c\} & \text{if } v=y_1, \\
L'(v) & \text{otherwise}.
\end{cases}\]
Suppose $r: \{y_1\} \rightarrow \bigcup_{v \in V(G)}L(v)$ is the request of $L$ given by $r(y_1)=c$.  Since there is no proper $L'$-coloring of $G'$, there is no proper $L$-coloring of $G$ that colors $y_1$ with $c$.  This means that none of the vertex requests can be satisfied.

If $K_{m,n-1}$ is $(k-1,k)$-choosable, but $K_{m-1,n}$ is not $(k,k-1)$-choosable, we can use a similar idea to construct a $k$-assignment $L$ for $G$ and a request $r$ of $L$ for which none of the vertex requests can be satisfied.  Thus, if $K_{m,n-1}$ is not $(k-1,k)$-choosable or $K_{m-1,n}$ is not $(k,k-1)$-choosable, then $\epsilon_{\ell}(G,k) =0$.
\end{proof}

We are now ready to prove Theorem~\ref{thm: K3n} which we restate.

\begin{customthm} {\bf \ref{thm: K3n}}
Suppose $G = K_{3,n}$ and $n \geq 3$ and $n\in \mathbb{N}$.  Then, the following statements hold. 
\begin{enumerate}[label=(\roman*)]
\item $\epsilon_{\ell}(G,t)= 1/2$ whenever $t\geq 4$.  
\item $\epsilon_{\ell}(G,3)= 1/2$ when $n=3$. 
\item $1/3 \leq \epsilon_{\ell}(G,3)\leq 1/2$ whenever $4 \leq n \leq 6$. 
\item $\epsilon_{\ell}(G,3) = 1/3$ whenever $7 \leq n \leq 8$. 
\item $\epsilon_{\ell}(G,3)= 0$ whenever $9\leq n \leq 26$.
\end{enumerate}
\end{customthm}

\begin{proof}
Throughout the proof suppose the bipartition of $G$ is $X=\{x_1,x_2,x_3\}$ and $Y=\{y_1,y_2,\ldots, y_n\}$.  First, notice that Statement~(i) follows immediately from Theorem~\ref{thm: awesome_sauce}.  

We now prove Statement~(ii).  Suppose $n=3$, $L$ is an arbitrary 3-assignment for $G$, and $r : D \rightarrow \bigcup_{v \in V(G)} L(v)$ is a request of $L$.  We will show there is a proper $L$-coloring of $G$ that satisfies at least $|D|/2$ vertex requests. First, suppose $ |D| \leq 2$. Without loss of generality, assume $r(x_1)=c_1$.  Let $L'$ be the list assignment for $G-x_1$ given by $L'(x_j)=L(x_j)$ if $j\in \{2,3\}$ and $L'(y_j)= L(y_j)- \{c_1\}$ if $j\in[3]$. By Theorem~\ref{thm: 3_2}, there is a proper $L'$-coloring of $G-x_1$. Thus, there is a proper $L$-coloring of $G$ that colors $x_1$ with $c_1$.

Now, suppose $3 \leq |D| \leq 4$. We will construct a proper $L$-coloring $f$ of $G$ that satisfies that at least two vertex requests. We know at least two elements of $D$ are in one of the partite sets. Without loss of generality, assume $r(x_1)=c_1$ and $r(x_2)=c_2$. By Lemma~\ref{lem: P3}, we know there is a proper $L$-coloring of $G$ that colors $x_i$ with $c_i$ for each $i \in [2]$ unless $L(y_1)=\{c_1,c_2,q\}$, $L(y_2)=\{c_1,c_2,r\}$, $L(y_3)=\{c_1,c_2,s\}$, and $L(x_3)=\{q,r,s\}$. So, we assume these four equalities are true. If $x_3 \in D$, let $f(x_3)=r(x_3)$, $f(x_2)=c_2$, and $f(y)=c_1$ for all $y \in Y$. To complete $f$, color $x_1$ with an element of $L(x_1)-\{c_1\}$.

If $x_3 \notin D$, then assume without loss of generality that $y_1 \in D$. Note that $r(y_1)\neq c_1$ or $r(y_1)\neq c_2$. If $r(y_1)\neq c_1$, let $f(y_1)=r(y_1)$, $f(x_1)=c_1$, and $f(y_2)=f(y_3)=c_2$. For the remaining uncolored vertices $x$ in $X$, let $f(x)$ be an element of $(L(x)-\{r(y_1),c_2\})$. If $r(y_1)\neq c_2$, let $f(y_1)=r(y_1)$, $f(x_2)=c_2$, and $f(y_2)=f(y_3)=c_1$. For each of the remaining uncolored vertices $x$ in $X$, let $f(x)$ be an element of $(L(x)-\{r(y_1),c_1\})$.

Finally, suppose $5 \leq |D| \leq 6$. Assume without loss of generality that $r(x_i)=c_i$ for $i \in [3]$, $r(y_1)=c_4$, and $r(y_2)=c_5$. Notice these assumptions mean $\{x_1,x_2,x_3,y_1,y_2\} \subseteq D$. We will construct a proper $L$-coloring of $G$, $f$, that satisfies at least $3$ vertex requests. Suppose that $c_1$, $c_2$, and $c_3$ are not pairwise distinct. Let $f(x_i)=c_i$ for all $i \in [3]$. Since there are $i,j \in [3]$ satisfying $i \neq j$ and $c_i=c_j$, $(L(y)-\{c_1,c_2,c_3\})$ is nonempty for each $y \in Y$. So, we can complete $f$ by letting $f(y)$ be an element of $L(y)-\{c_1,c_2,c_3\}$ for each $y \in Y$. So, we may assume that $c_1$, $c_2$, and $c_3$ are pairwise distinct.

Notice in the case that $L(y) \neq \{c_1,c_2,c_3\}$ for each $y \in Y$.  We can construct a desired coloring by letting $f(x_i)=c_i$ for all $i \in [3]$, and letting $f(y)$ be an element of $L(y)-\{c_1,c_2,c_3\}$ for each $y \in Y$. Therefore, we may assume $L(y)=\{c_1,c_2,c_3\}$ for some $y \in Y$.

Now, assume without loss of generality that $c_4 \notin \{c_1,c_2\}$.   Let $f(x_1)=c_1$, $f(x_2)=c_2$, and $f(y_1)=c_4$. By Lemma~\ref{lem: P3}, we can complete $f$ so that it has the desired properties unless $L(y_2)=\{c_1,c_2,q\}$, $L(y_3) = \{c_1,c_2,r\}$, and $L(x_3)=\{c_4,q,r\}$. Thus, we assume these three equalities are true.  Since $c_3 \in \{c_4,q,r\}$, we have that $c_3 = q$, $c_3 = r$, or $c_3 = c_4$.  We will construct a proper $L$-coloring $f$ of $G$ satisfying at least 3 vertex requests in each of these cases.

If $c_3 \in \{q,r\}$, let $f(x_1)=c_1$, $f(x_3)=c_3$, $f(y_1)=c_4$, $f(y_2) = c_2$, and $f(y_3)=c_2$.  Finally let $f(x_2)$ be an element of $L(x_2) - \{c_2,c_4\}$.

Now, assume that $c_3 = c_4$.  Then, since $L(y)=\{c_1,c_2,c_3\}$ for some $y \in Y$, we have that $L(y_1) = \{c_1,c_2,c_3\}$.  Since $c_5 \in \{c_1,c_2,q\}$, we may suppose that some fixed $j \in [2]$ satisfies $c_j \in \{c_1,c_2\} - \{c_5\}$.  Let $j'$ be such that $\{j,j'\} = [2]$.  Let $f(x_j)=c_j$, $f(x_3)=c_3$, $f(y_1)=c_{j'}$, $f(y_2) = c_5$, and $f(y_3)=c_{j'}$.  Finally color $x_{j'}$ with an element of $L(x_{j'}) - \{c_{j'},c_5\}$.

Now we turn our attention to Statements~(iii) and~(iv).  First, Proposition~9 in~\cite{CH25} combined with Proposition 14 in~\cite{KM23} implies that $1/3 \leq \epsilon_{\ell}(G,3)$ when $n \in \{4,5,6,7,8\}$.  When $n \in \{4,5,6\}$, we clearly have $\epsilon_{\ell}(G,3) \leq 1/ \rho(G) = 1/2$ which completes the proof of Statement~(iii).

To complete the proof of Statement~(iv), suppose $n \in \{7,8\}$.  Let $L$ be the $3$-assignment for $G$ given by $L(x_1)=\{1,4,5\}$, $L(x_2)=\{2,6,7\}$, $L(x_3)=\{3,8,9\}$, $L(y_1)=[3]$, $L(y_2)=\{1,2,8\}$, $L(y_3)=\{1,2,9\}$, $L(y_4)=\{1,3,6\}$, $L(y_5)=\{1,3,7\}$, $L(y_6)=\{2,3,4\}$, and $L(y_7)=\{2,3,5\}$. If $n=8$, then let $L(y_8) = [3]$. Additionally, suppose that that $r$ is the request of $L$ with domain $\{x_1,x_2,x_3\}$ given by $r(x_i) = i$ for each $i \in [3]$. We claim that there is no proper $L$-coloring of $G$ that satisfies at least two of the vertex requests (consequently $\epsilon_{\ell}(G,3) \leq 1/3$). 

Assume for the sake of contradiction that $f$ is a proper $L$-coloring of $G$ satisfying at least 2 vertex requests. This means $f(x_i) = i$ for at least two $i \in [3]$.  Suppose $f(x_1)=1$ and $f(x_2)=2$. Then, $f(X) \in \{L(y_1),L(y_2),L(y_3)\}$. Second, suppose $f(x_1)=1$, and $f(x_3)=3$. Then, $f(X) \in \{L(y_1),L(y_4),L(y_5)\}$. Third, suppose $f(x_2)=2$, and $f(x_3)=3$. Then, $f(X) \in \{L(y_1),L(y_6),L(y_7)\}$. In each case we get a contradiction that $f$ is a proper $L$-coloring. Therefore, $G$ is not $L$-colorable. 

Finally, we prove Statement (v) suppose $n \in \{9,10, \ldots, 26\}$.  Note that by Theorem~\ref{thm: 3_2}, $K_{2,n}$ is not $(3,2)$-choosable.  So, by Lemma~\ref{lem: connect} we have that $\epsilon_{\ell}(G,3)  = 0$ as desired.
\end{proof}

Finally, we prove Proposition~\ref{pro: forfree}. Recall for $3 \leq m \leq n$, $K_{m,n}$ is 3-choosable if and only if $m=3$ and $n \leq 26$, $m=4$ and $n \leq 20$, $m=5$ and $n \leq 12$, or $m=6$ and $n \leq 10$.

\begin{custompro} {\bf \ref{pro: forfree}}
    Suppose $G=K_{m,n}$ and $m \leq n$. Then, 
    $\epsilon_{\ell}(G,3) = 0$ when $m=4$ and $7 \leq n \leq 20$, $m=5$ and $n \leq 12$, or $m=6$ and $n \leq 10$.  When $m=4$ and $n \leq 6$, $1/|V(G)| \leq \epsilon_{\ell}(G,3) \leq 1/2$.
\end{custompro}

\begin{proof}
First, suppose $m=4$ and $7 \leq n \leq 20$, $m=5$ and $n \leq 12$, or $m=6$ and $n \leq 10$.  Notice this means $G$ is 3-choosable.  Then, by Theorem~\ref{thm: 3_2} $K_{m-1,n}$ is not $(3,2)$-choosable. So, by Lemma~\ref{lem: connect} we have $\epsilon_{\ell}(G,3) = 0$.

Second, suppose $m=4$ and $n\leq6$.  Clearly $G$ is 3-choosable.  Note that by Theorem~\ref{thm: 3_2}, $K_{4,n-1}$ is $(2,3)$-choosable and $K_{3,n}$ is $(3,2)$-choosable.  So, by Lemma~\ref{lem: connect} we have that $\epsilon_{\ell}(G,3) \geq 1/(4+n) = 1/|V(G)|$.  We also immediately have $\epsilon_{\ell}(G,3) \leq 1/\rho(G) = 1/2$.  
\end{proof}

\noindent{\bf Acknowledgment:} The authors would like to thank Doug West for his encouragement. The authors would also like to thank Hemanshu Kaul, Michael Pelsmajer, and the anonymous reviewers for their helpful comments.  This project was completed by the Coloring Research Group of South Alabama at the University of South Alabama during the spring 2024, summer 2024, and fall 2024 semesters.  The support of the University of South Alabama is gratefully acknowledged.


\begin{thebibliography}{99}
{\small

\bibitem{AC21} N. Alon, S. Cambie, R. J. Kang, Asymmetric list sizes in bipartite graphs, \emph{Annals of Combinatorics}, 25 (2021), 913-933.

\bibitem{BL22} A. Blumenthal, B. Lidick\'y, R. R. Martin, S. Norin, F. Pfender, J. Volec, Counterexamples to a conjecture of Harris on Hall ratio, \emph{SIAM J. of Discrete Mathematics} 36(3) (2022).

\bibitem{B77} O. Borodin, Criterion of chromaticity of a degree prescription, \emph{Abstracts of IV All-Union Conf. on Th. Cybernetics}, (1977), 127-128.

\bibitem{BM22} P. Bradshaw, T. Masa\v{r}\'ik, L. Stacho, Flexible list colorings in graphs with special degeneracy conditions, \emph{J. Graph Theory}, 101 (2022), 717-745.

\bibitem{CC21} S. Cambie, W. Cames van Batenburg, E. Davies, R. J. Kang, Packing list-colourings, \emph{Random Structures \& Algorithms}, 64(1) (2024), 62-93.

\bibitem{CH25} S. Cambie, R. Hämäläinen, Packing colourings in complete bipartite graphs and the inverse problem for correspondence packing, \emph{J. Graph Theory} 109(1) (2025), 52–61.

\bibitem{CR12} D. W. Cranston, L. Rabern, List-coloring claw-free graphs with $\Delta-1$ colors, \emph{SIAM Journal on Discrete Mathematics}, 31(2) (2017), 726-748.

\bibitem{CJ00} M. Cropper, M. S. Jacobson, A. Gy{\'a}rf{\'a}s, J. Lehel, The Hall ratio of graphs and hypergraphs, \emph{Les cahiers du laboratorie}, No. 17 (2000).

\bibitem{DM20} Z. Dvo\v{r}\'{a}k, P. O. Mendez, H. Wu, 1-subdivisions, the fractional chromatic number, and the Hall ratio, \emph{Combinatorica}, 40 (2020), 759-774.

\bibitem{DN19} Z. Dvo\v{r}\'{a}k, S. Norin, L. Postle, List coloring with requests, \emph{J. Graph Theory} 92 (2019), 191-206.

\bibitem{ET79} P. Erd\H{o}s, A. L. Rubin, H. Taylor, Choosability in graphs, \emph{Congressus Numerantium} 26 (1979), 125-127.

\bibitem{HJ97}  A.J.W. Hilton, P.D. Johnson Jr., D. A. Leonard, Hall’s condition for multicolorings, \emph{Congressus Numerantium}, 128 (1997), 195-203.

\bibitem{KM23} H. Kaul, R. Mathew, J. A. Mudrock, M. J. Pelsmajer, Flexible list colorings: Maximizing the number of requests satisfied, \emph{J. Graph Theory} 106(4) (2024), 887-906.

\bibitem{O95} P. O'Donnell, The choice number of $K_{6,q}$, Rutgers University Mathematics Department (preprint), 1995.

\bibitem{V76} V. G. Vizing, Coloring the vertices of a graph in prescribed colors, \emph{Diskret. Analiz.} no. 29, \emph{Metody Diskret. Anal. v Teorii Kodovi Skhem} 101 (1976), 3-10.

\bibitem{W01} D. B. West, (2001) \emph{Introduction to Graph Theory, 2nd edition}.  Upper Saddle River, NJ: Prentice Hall.

\bibitem{W20} D. B. West, (2020) \emph{Combinatorial Mathematics}.  New York, NY: Cambridge University Press.

}


\end{thebibliography}
\end{document}